\documentclass{article}

\usepackage{amsfonts, amsmath, amssymb, amsthm}

\usepackage{tikz}
\usetikzlibrary{intersections}

\DeclareMathAlphabet{\mathpzc}{OT1}{pzc}{m}{it}

\newtheorem{thm}{Theorem}
\newtheorem{cor}[thm]{Corollary}
\newtheorem{lem}[thm]{Lemma}

\newtheorem{defn}[thm]{Definition}

\newcommand{\ff}[1]{{\mathbb F}_{#1}}
\newcommand{\ffs}[1]{{\mathbb F}_{#1}^\star}
\newcommand{\ffx}[1]{\ff{#1}[X]}
\newcommand{\ffxy}[1]{\ff{#1}[X,Y]}
\newcommand{\ffxyz}[1]{\ff{#1}[X,Y,Z]}

\newcommand{\ffxi}[2]{
  \ifnum#2=0{ {\mathbb F}_{#1}[X_1,\ldots,X_n] }
  \else{
    \ifnum#2=1{ \ffx{#1} }
    \else{ \ifnum#2=2{ {\mathbb F}_{#1}[X_1,X_2] }
      \else{ {\mathbb F}_{#1}[X_1,\ldots,X_{#2}] }
           \fi}
    \fi}
  \fi}

\DeclareMathOperator{\ptx}{\bf x}
\DeclareMathOperator{\pty}{\bf y}
\DeclareMathOperator{\pto}{\bf O}
\DeclareMathOperator{\pti}{\bf I}
\DeclareMathOperator{\ptj}{\bf J}
\DeclareMathOperator{\pta}{\bf p}
\DeclareMathOperator{\ptb}{\bf q}
\DeclareMathOperator{\ptc}{\bf r}
\newcommand{\lne}{\mathcal L}
\newcommand{\mne}{\mathcal M}
\newcommand{\pplane}{\mathpzc P}
\newcommand{\aplane}{\mathpzc A}
\newcommand{\splane}{\mathpzc S}

\newcommand{\K}{\mathcal K}

\newcommand{\T}{\mathcal T}
\newcommand{\R}{\mathcal R}
\newcommand{\Sset}{\mathcal S}
\newcommand{\Rs}{{\mathcal R}^\star}

\begin{document}
\title{On coordinatising planes of prime power order using finite fields}

\author{Robert S. Coulter\\
\\
520 Ewing Hall\\
Department of Mathematical Sciences\\
University of Delaware\\
Newark, DE, USA, 19716}

\maketitle

\begin{abstract}
We revisit the coordinatisation method for projective planes.
First, we discuss how the behaviour of the additive and multiplicative
loops can be described in terms of its action on the ``vertical" line, and
how this means one can coordinatise certain planes in an optimal sense.
We then move to consider projective planes of prime power order only.
Specifically, we consider how coordinatising planes of prime power order using
finite fields as the underlying labelling set leads to some general
restrictions on
the form of the resulting planar ternary ring (PTR) when viewed as a trivariate
polynomial over the field.
We also consider the Lenz-Barlotti type of the plane being coordinatised,
deriving further restrictions on the form of the PTR polynomial.
\end{abstract}

\section{Introduction}

This paper is concerned with two interlinked areas in the study of
projective planes -- namely the coordinatisation method and the Lenz-Barlotti
classification.
The coordinatisation method takes an arbitrary projective plane and produces a
trivariate function known as a {\em planar ternary ring (PTR)} over whatever
set is used as the labelling set during the coordinatisation process.
The Lenz-Barlotti (LB) classification is a coarse classification system for
affine and projective planes centred on the transitive behaviour exhibited by
the full automorphism group of the plane.

We begin by outlining the coordinatisation method using slightly non-standard
diagrams, and describe how this leads to the concept of a planar ternary ring
(PTR).
From the PTR so constructed it is common to define an ``additive" and a
``multiplicative" loop.
Through the use of our diagrams, we can give an explicit description of
the action of these loops on the vertical line.
We have dual motivations in this initial discussion.
In the long term, our motivation is a desire to give a meaningful definition of
``optimal coordinatisation", and this is achieved through our understanding of
these actions.
In the short term, our motivation is to give an additional insight into a well
known conjecture in projective geometry concerning Fano configurations.
It has been known for some time that a finite Desarguesian plane contains a
Fano configuration if and only if the plane has even order.
For non-Desarguesian planes, a folk-lore conjecture claims that all
finite non-Desarguesian planes must contain a Fano configuration.
(Though the conjecture has been attributed to Hanna Neumann, she did not
make the conjecture.)
In support of the conjecture, several classes of planes have been shown to
contain Fano configurations; for a non-exhaustive set of examples see Neumann
\cite{neu54}, Rahilly \cite{rah73}, Johnson \cite{joh07}, or Petrak
\cite{pet10}.
Here we note that the action of the additive loop on the vertical line provides
an obvious necessary and sufficient condition for the existence of a Fano
configuration in a projective plane, though the utility of these conditions to
prove the conjecture is unclear.

We then concentrate on the coordinatisation of projective planes of prime power
order.
(Of course, anyone who believes the prime-power conjecture is true would view
this as no restriction at all; the present author is not willing to express any
view on that conjecture's validity, at least not in print!)
Specifically, we here instigate a study of projective planes of prime power
order via their coordinatisation over finite fields of the appropriate order.
In this way we are able to view the resulting PTR as a reduced trivariate
polynomial over a finite field, what we call a {\em PTR polynomial}.
We then derive restrictions on the form of the PTR polynomial using the
functional properties that any PTR must exhibit.
As shall be seen, several forms of reduced permutation polynomials and
$\kappa$-polynomials (both of which we shall define below) naturally arise from
this relation.
The culmination of the results of this section, and the main statement in this
general situation, is given in Theorem \ref{ptrpolyform}.

Finally, we outline the Lenz-Barlotti classification system for
projective planes.
It is generally well known that knowledge of the Lenz-Barlotti type of a
projective plane $\pplane$ can lead to additional algebraic properties of the
PTR obtained from coordinatising it, but this only occurs when some effort is
made to coordinatise the plane in an optimal way.
We make explicit what we mean by optimal coordinatisation, and utilise
this concept to obtain further restrictions on the form of the PTR polynomial
under various assumptions concerning the LB type.
In particular, we show how one can coordinatise suitable planes so that
either the additive or multiplicative loop resulting from the coordinatisation
is exactly the same as its corresponding field operation, and consider how
this can affect the form of the PTR polynomial.
Theorems \ref{realaddition} and \ref{realmultiplication} are the main results
of this section.

The paper is set out as follows.
In Section \ref{coordmethod}, we give an overview of the coordinatisation
approach,
and discuss the actions of the loops on the vertical line. There we also
discuss the conjecture on the existence of Fano configurations.
In Section \ref{fqcoord} we restrict the coordinatisation process to planes of
prime power order and where the coordinatising set used is a finite field,
and describe explicitly how a PTR polynomial is produced.
Section \ref{ptrpolys} then provides a sequence of results on the behaviour and
form of PTR polynomials.
In the final section, we turn to a discussion of the Lenz-Barlotti
classification for projective planes and affine planes.
There we make explicit the concept of an ``optimal" coordinatisation of a
plane, and exploit this idea to produce further restrictions on the PTR
polynomial based on knowledge of the LB type of the projective plane when
coordinatised optimally.

\section{Coordinatisation} \label{coordmethod}

The method of coordinatisation has been used now for over seventy years.
There are at least 3 standard coordinatisation methods. Though they are all
essentially equivalent, they produce slightly different properties in the
resulting PTRs.
For the sake of consistency, we shall use the process outlined by Hughes and
Piper in \cite{bhughes73}, Chapter 5 -- they give the two other methods at the
end of that same chapter.
In this section we will describe precisely the coordinatisation method for
introducing a coordinate system for an abstract projective plane.
While there are several readily available sources for describing this method,
our motivation for providing another treatment is twofold: firstly, there is
the desire for a self-contained discussion, and secondly,
we will use diagrams which are not standard elsewhere with the explicit
aim of making it easier to visualise certain concepts we wish to discuss.

Let $\pplane$ be a projective plane of order $n$ and let $\R$ be any set of
cardinality $n$ -- this set along with the symbol $\infty$ will be all that is
required to produce a coordinate system for the plane.
We designate two special elements of $\R$ by $0$ and $1$ for reasons which will
become clear. We now proceed to coordinatise $\pplane$.
\begin{itemize}
\item Choose any triangle in the plane $\pto,\ptx,\pty$.
Label $\pto=(0,0)$, $\ptx=(0)$ and $\pty=(\infty)$ -- by doing so we have now
determined the ``line at infinity" $\overline{\ptx\pty}=[\infty]$.
We also set $[0]=\overline{\pto\pty}$ and $[0,0]=\overline{\pto\ptx}$.
(The process for an affine plane $\aplane$ differs from the projective version
only in that the line at infinity is pre-determined, so that the choice of
points $\ptx$ and $\pty$ is restricted.)

\item A fourth point, $\pti$, not collinear with any two of $\pto, \ptx, \pty$
is now chosen and labelled $\pti=(1,1)$.

\item To finalise the initialisation process, we label some obvious
intersection points:
\begin{itemize}
\item Set $\overline{\ptx\pti}\cap [0] =(0,1)$.
\item Set $\overline{\pty\pti}\cap [0,0] =(1,0)$.
\item Set $\overline{(1,0)(0,1)}\cap [\infty]= \ptj=(1)$.
\end{itemize}
\end{itemize}
The situation after this initial phase is given in Figure 1.
\begin{center}
\begin{tikzpicture}[domain=-2:12]
\coordinate 	(O) at (0,0);
\coordinate 	(X) at (10,0);
\coordinate 	(Y) at (0,5);
\coordinate 	(Zero1) at (0,1.5);
\coordinate 	(One0) at (2,0);

\draw		(O) circle [radius=3pt] node[anchor=north east]{$\pto$};
\draw		(X) circle [radius=3pt] node[anchor=north]{$\ptx=(0)$};
\draw		(Y) circle [radius=3pt] node[anchor=east]{$\pty=(\infty)$};
\draw[fill]	(Zero1) circle [radius=3pt] node[anchor=east]{$(0,1)$};
\draw[fill]	(One0) circle [radius=3pt] node[anchor=north]{$(1,0)$};

\draw[thick]	(-1,0) -- (11,0);
\draw[thick]	(0,-1) -- (0,6);

\draw[thick, name path=XY]	(Y) .. controls (14/3,14/3) and (8,3) .. (X);

\path[name path=Zero1J] (Zero1) parabola bend (One0) (8.5,2.2);

\draw[dashed, very thick, name intersections={of=XY and Zero1J, by={J}}]
	(Zero1) parabola bend (One0) (J);

\draw[fill]	(J) circle [radius=3pt] node[anchor=south west]{$\ptj=(1)$};

\draw[dashed, very thick, name path=Zero1X]	(X) .. controls (3,3) and (2,2.5) .. (Zero1);
\draw[dashed, very thick, name path=One0Y]	(Y) .. controls (2.7,3) and (2.5,1.5) .. (One0);

\draw[name intersections={of=Zero1X and One0Y, by={I}}]
		(I) circle [radius=3pt] node[anchor=south west]{$\pti=(1,1)$};

\end{tikzpicture}

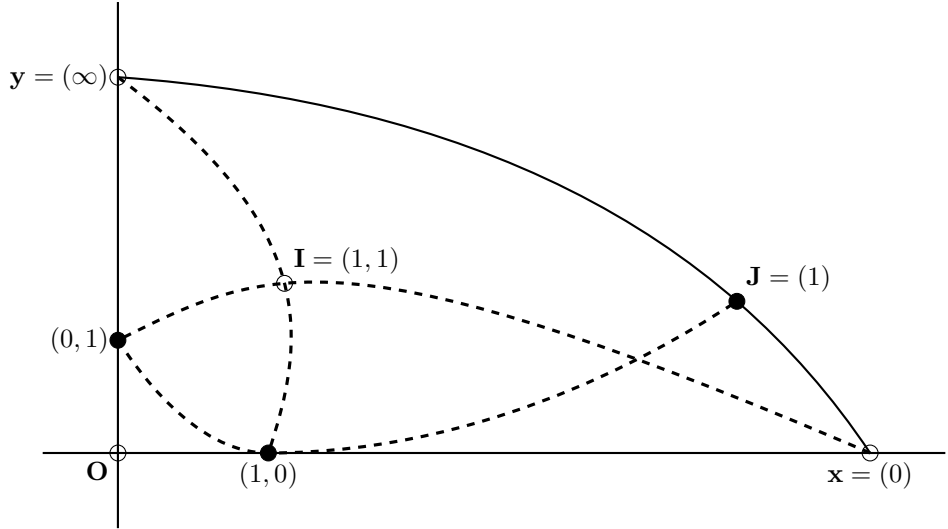
Figure 1: After the initial labelling.
\end{center}
At this point, we have labelled 3 of the $n+1$ points of both of the lines
$[0]$ and $[0,0]$.
One may now label the remaining $n-2$ points of $[0]$ as $(0,a)$ in arbitrary
way using the remaining $n-2$ elements $a\in\R\setminus\{0,1\}$.
This is the last remaining freedom of choice in the process, as from this stage
onwards, the coordinates of all points and lines are totally determined.
Later in this section we will explain how
the additive and multiplicative loops that result from the coordinatising
procedure can be seen to act on $\overline{\pto\pty}$, thus outlining how the
elements of $\R$ interact under these operations follows from this random
labelling.

We now proceed to label all points and lines of the plane; see Figure 2.
\begin{itemize}
\item To label the remaining points of $[0,0]$ we set
$\overline{(0,a)\ptj}\cap [0,0] = (a,0)$.
\item To label the remaining points of $[\infty]$ we set
$\overline{(0,a)(1,0)}\cap [\infty] = (a)$.
\item To label the remaining ``affine" points we set
$\overline{(a,0)\pty)}\cap\overline{(0,b)\ptx}=(a,b)$.
\end{itemize}
\begin{center}
\begin{tikzpicture}[domain=-2:12]
\coordinate 	(O) at (0,0);
\coordinate 	(X) at (10,0);
\coordinate 	(Y) at (0,5);
\coordinate 	(Zero1) at (0,1.5);
\coordinate 	(One0) at (2,0);
\coordinate	(Zeroa) at (0,2);
\coordinate	(Zerob) at (0,3.5);
\coordinate	(A0) at (5,0);

\draw		(O) circle [radius=3pt] node[anchor=north east]{$\pto$};
\draw		(X) circle [radius=3pt] node[anchor=north]{$\ptx=(0)$};
\draw		(Y) circle [radius=3pt] node[anchor=east]{$\pty=(\infty)$};
\draw		(One0) circle [radius=3pt] node[anchor=north]{$(1,0)$};
\draw[fill]	(A0) circle [radius=3pt] node[anchor=north]{$(a,0)$};
\draw		(Zeroa) circle [radius=3pt] node[anchor=east]{$(0,a)$};
\draw		(Zerob) circle [radius=3pt] node[anchor=east]{$(0,b)$};

\draw[thick]	(-1,0) -- (11,0);
\draw[thick]	(0,-1) -- (0,6);

\draw[thick, name path=XY]	(Y) .. controls (14/3,14/3) and (8,3) .. (X);

\path[name path=Zero1J] (Zero1) parabola bend (One0) (8.5,2.2);

\path[name path=ZeroaA] (Zeroa) parabola bend (One0) (5,4.5);

\path[thick, name intersections={of=XY and Zero1J, by={J}}]
	(Zero1) parabola bend (One0) (J);

\draw[dashed, very thick, name intersections={of=XY and ZeroaA, by={A}}]
	(Zeroa) parabola bend (One0) (A);

\draw		(J) circle [radius=3pt] node[anchor=south west]{$\ptj=(1)$};

\draw[fill]	(A) circle [radius=3pt] node[anchor=south west]{$(a)$};

\draw[dashed, very thick, name path=ZerobX]	(Zerob) .. controls (11/3,11/3) and (7,2) .. (X);

\draw[dashed, very thick, name path=A0Y]	(Y) .. controls (13/6,23/6) and (23/6,13/6) .. (A0);
		
\draw[fill, name intersections={of=A0Y and ZerobX, by={AB}}]
		(AB) circle [radius=3pt] node[anchor=south west]{$(a,b)$};
		
\draw[dashed, very thick]	(Zeroa) parabola bend (A0) (J);

\end{tikzpicture}

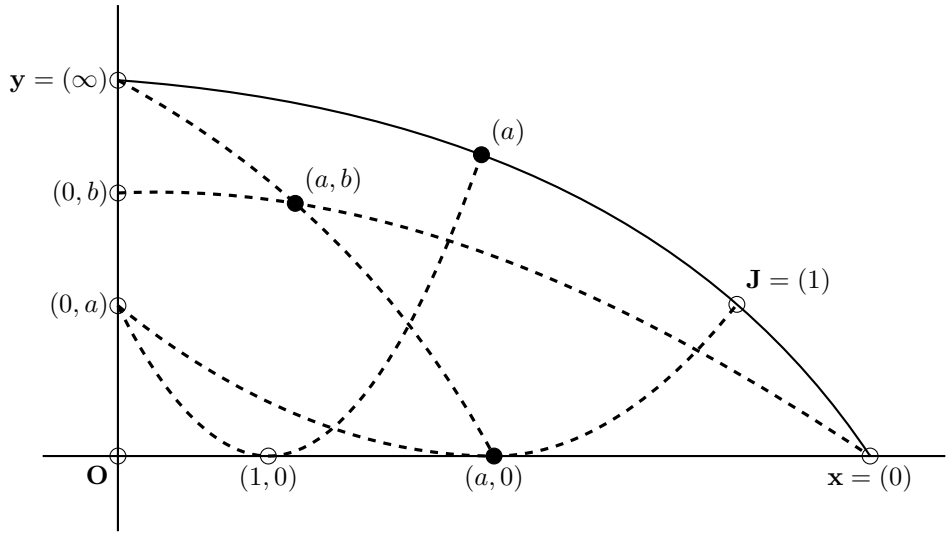
Figure 2: Point labelling.
\end{center}
With a labelling of the points complete, it remains only to give a labelling
of the lines (Figure 3).
\begin{itemize}
\item To label the ``vertical'' lines we set $\overline{(a,0)\pty} = [a]$.
\item To label the ``lines of slope $m$" we set $\overline{(m)(0,k)}= [m,k]$.
\end{itemize}
\begin{center}
\begin{tikzpicture}[>=stealth,domain=-2:12]
\coordinate 	(O) at (0,0);
\coordinate 	(X) at (10,0);
\coordinate 	(Y) at (0,5);
\coordinate 	(Zero1) at (0,1.5);
\coordinate 	(One0) at (2,0);
\coordinate	(Zeroa) at (0,2);
\coordinate	(Zerob) at (0,3.5);
\coordinate	(A0) at (5,0);
\coordinate	(Zerok) at (0,2.5);

\draw		(O) circle [radius=3pt] node[anchor=north east]{$\pto$};
\draw		(X) circle [radius=3pt] node[anchor=north]{$\ptx=(0)$};
\draw		(Y) circle [radius=3pt] node[anchor=east]{$\pty=(\infty)$};
\draw		(A0) circle [radius=3pt] node[anchor=north]{$(a,0)$};
\draw		(Zerok) circle [radius=3pt] node[anchor=east]{$(0,k)$};

\draw[thick]	(-1,0) -- (11,0);
\draw[thick]	(0,-1) -- (0,6);
\draw[thick, name path=XY]	(Y) .. controls (14/3,14/3) and (8,3) .. (X);

\path[name path=ZerokM] (Zerok) parabola bend (Zerok) (7,4);

\draw[very thick, dashed, name intersections={of=XY and ZerokM, by={M}}]
	(Zerok) parabola bend (Zerok) (M);

\draw		(M) circle [radius=3pt] node[anchor=south west]{$(m)$};

\draw[very thick, dashed, name path=A0Y]
	(Y) .. controls (13/6,23/6) and (23/6,13/6) .. (A0);

\path[name path=Path1]	(8,2) -- (10,2);
		
\draw[thick,->,name intersections={of=XY and Path1, by={pt1}}]
	(8.8,2.4) node[anchor=west]{$[\infty]$} -- (pt1);

\draw[thick,<-]	(2,0) -- (2,-0.7) node[anchor=north]{$[0,0]$};

\draw[thick,<-]	(0,1) -- (-0.7,1) node[anchor=east]{$[0]$};

\path[name path=Path2]	(5,1.2) -- (4,1.2);

\draw[thick,<-, name intersections={of=Path2 and A0Y,by={pt2}}]
	(pt2) -- (5,1.2) node[anchor=west]{$[a]$};

\path[name path=Path3]	(4.7,2) -- (4.7,4);

\draw[thick,<-, name intersections={of=Path3 and ZerokM,by={pt3}}]
	(pt3) -- (5.2,2.7) node[anchor=west]{$[m,k]$};

\end{tikzpicture}

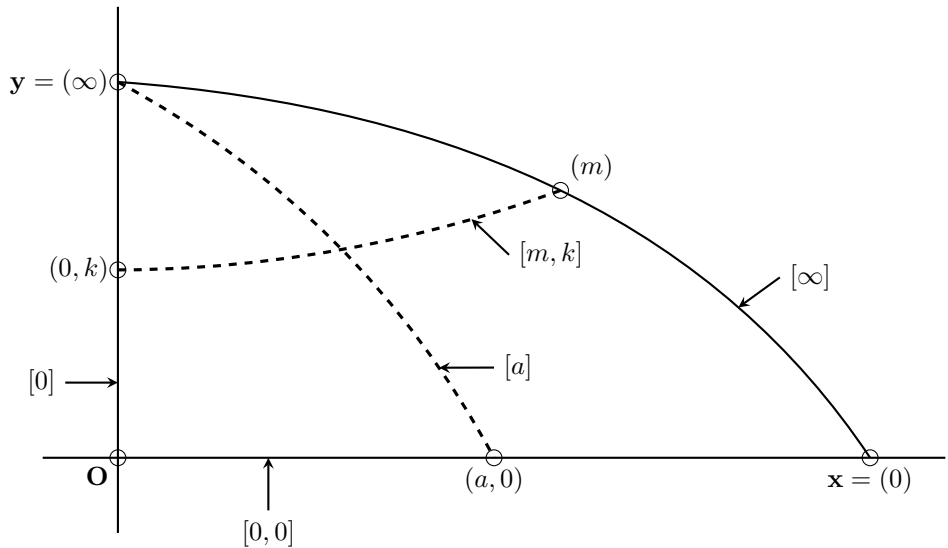
Figure 3: Line labelling.
\end{center}
From this coordinatisation, one now defines a tri-variate function $T$ on $\R$,
called a {\em planar ternary ring (PTR)}, by setting $T(m,x,y)=k$ if and only
if $(x,y)\in [m,k]$.
This PTR will exhibit certain properties and is actually
equivalent to the projective plane as any three variable function exhibiting
those properties can be used to define a projective plane.
More precisely, we have the following important result, essentially due to Hall
\cite{hall43}; see also Hughes and Piper, \cite{bhughes73}, Theorem 5.1.
\begin{lem}[Hall, \cite{hall43}, Theorem 5.4] \label{hplem}
Let $\pplane$ be a projective plane of $n$ and $\R$ be any set of cardinality
$n$.
Let $T:\R^3\rightarrow \R$ be a PTR obtained from coordinatising $\pplane$.
Then $T$ must satisfy the following properties:
\begin{enumerate}
\renewcommand{\labelenumi}{(\alph{enumi})}
\item $T(a,0,z)=T(0,b,z)=z$ for all $a,b,z\in\R$.
\item $T(x,1,0)=x$ and $T(1,y,0)=y$ for all $x,y\in\R$.
\item If $a,b,c,d\in\R$ with $a\ne c$, then there exists a unique
$x$ satisfying $T(x,a,b)=T(x,c,d)$.
\item If $a,b,c\in\R$, then there is a unique $z$ satisfying
$T(a,b,z)=c$.
\item If $a,b,c,d\in\R$ with $a\ne c$, then there is a unique pair
$(y,z)$ satisfying $T(a,y,z)=b$ and $T(c,y,z)=d$.
\end{enumerate}
Conversely, any tri-variate function $T$ defined on $\R$ which satisfies
Properties (c) through (e)
can be used to define an affine plane $\aplane_T$ of order $q$ as
follows:
\begin{itemize}
\item the points of $\aplane$ are $(x,y)$, with $x,y\in\R$;
\item the lines of $\aplane$ are the symbols $[m,a]$, with $m,a\in\R$,
defined by
\begin{equation*}
[m,a] = \{ (x,y)\in\R\times\R \,:\, a = T(m,x,y) \},
\end{equation*}
and the symbols $[c]$, with $c\in\R$, defined by
\begin{equation*}
[c] = \{ (c,y) \,:\, y\in\R \}.
\end{equation*}
\end{itemize}
\end{lem}
Since one only needs Properties (c) through (e) to construct $\pplane$, a
polynomial satisfying just the latter three properties is called a {\em weak
PTR}.
If a weak PTR also satisfies (a) (resp. (b)), then it is a
{\em weak PTR with zero (resp. weak PTR with unity)}.

It is customary to define an addition $\oplus$ and multiplication $\odot$ by
\begin{align*}
x\oplus y &= T(1,x,y),\\
x\odot y &= T(x,y,0),
\end{align*}
for all $x,y\in\R$.
It is well known that the properties of the plane guarantee that both
$\oplus$ and $\odot$ are loops with identities 0 and 1 over $\R$ and $\Rs$,
respectively.
A PTR is called {\em linear} over $\R$ if $T(x,y,z)=(x\odot y)\oplus z$
for all $x,y,z\in\R$ -- that is, if $T$ can be reconstructed from only knowing
the operations $\oplus$ and $\odot$.
One point of interest here is how the operations $\oplus$ and $\odot$ act on
the vertical line $[0]=\overline{\pto\pty}$; we shall outline this action
directly.
Before doing so, we mention an important example.
Consider the polynomial $T(X,Y,Z)=XY+Z$.
It is easily checked that the polynomial $T$ is a linear PTR over
any field $\K$; it defines the Desarguesian plane in every case.
It cannot be over emphasised that the same plane can yield many different PTRs
as choosing different quadrangles as the reference points $\pto, \ptx, \pty$
and $\pti$, may yield very different PTRs.
We discuss this further in Section \ref{LBsectionpt1}.

\subsection{The action of $(\R,\oplus)$ on $\overline{\pto\pty}$}

Let us first consider $(\R,\oplus)$. The process is anchored by our initial
triangle $\pto,\ptx,\pty$ and the point $\ptj=(1)$.
\begin{itemize}
\item Choose two points $(0,a),(0,b)$ on $\overline{\pto\pty}=[0]$.

\item Create the point $(a,0)=\overline{(0,a)\ptj}\cap \overline{\pto\ptx}$.

\item Next create the point
$(a,b)=\overline{(a,0)\pty)}\cap\overline{(0,b)\ptx}$.

\item Now consider the point
$(0,k)=\overline{\ptj(a,b)}\cap \overline{\pto\pty}$.
It lies on the line $[1,k]$ by construction. Furthermore, from the definition
of the PTR and $\oplus$ we see $k=T(1,a,b)=a\oplus b$.
Thus $(0,k)=(0,a\oplus b)$.
\end{itemize}
Pictorially, the action of $\oplus$ on the vertical line is seen in Figure 4.
\begin{center}
\begin{tikzpicture}[domain=-2:12]
\coordinate 	(O) at (0,0);
\coordinate 	(X) at (10,0);
\coordinate 	(Y) at (0,5);
\coordinate	(Zeroa) at (0,1);
\coordinate	(Zerob) at (0,3.5);
\coordinate	(A0) at (5,0);
\coordinate	(Zerok) at (0,2.5);

\draw		(O) circle [radius=3pt] node[anchor=north east]{$\pto$};
\draw		(X) circle [radius=3pt] node[anchor=north]{$\ptx=(0)$};
\draw		(Y) circle [radius=3pt] node[anchor=east]{$\pty=(\infty)$};
\draw		(A0) circle [radius=3pt] node[anchor=north]{$(a,0)$};
\draw		(Zeroa) circle [radius=3pt] node[anchor=east]{$(0,a)$};
\draw		(Zerob) circle [radius=3pt] node[anchor=east]{$(0,b)$};
\draw[fill]	(Zerok) circle [radius=3pt] node[anchor=east]{$(0,a\oplus b)$};

\draw[thick]	(-1,0) -- (11,0);
\draw[thick]	(0,-1) -- (0,6);

\draw[thick, name path=XY]	(Y) .. controls (14/3,14/3) and (8,3) .. (X);

\path[name path=Zero1J]		(Zero1) parabola bend (One0) (8.5,2.2);

\path[thick, name intersections={of=XY and Zero1J, by={J}}]
	(Zero1) parabola bend (One0) (J);

\draw		(J) circle [radius=3pt] node[anchor=south west]{$\ptj=(1)$};

\draw[dotted, thick, name path=ZerobX]
	(Zerob) .. controls (11/3,11/3) and (7,2) .. (X);

\draw[dotted, thick, name path=A0Y]
	(Y) .. controls (13/6,23/6) and (23/6,13/6) .. (A0);
		
\draw[name intersections={of=A0Y and ZerobX, by={AB}}]
	(AB) circle [radius=3pt] node[anchor=south west]{$(a,b)$};
		
\draw[dotted, thick]	(Zeroa) parabola bend (A0) (J);

\draw[dashed, very thick]	(Zerok) parabola bend (AB) (J);

\end{tikzpicture}

Figure 4: Action of the additive loop on the vertical line.
\end{center}

\subsection{The action of $(\Rs,\odot)$ on $\overline{\pto\pty}$}

As with the operation $\oplus$, the process by which the action of
$(\Rs,\odot)$ on $\overline{\pto\pty}$ is described relies on our initial
triangle $\pto,\ptx,\pty$ and the point $\ptj=(1)$. We'll also need the point
$(1,0)$.
\begin{itemize}
\item Choose two points $(0,a),(0,b)$ on $\overline{\pto\pty}=[0]$.

\item Create the point $(b,0)=\overline{(0,b)\ptj}\cap \overline{\pto\ptx}$.

\item Create the point $(a)=\overline{(0,a)(1,0)}\cap \overline{\ptx\pty}$.

\item Now consider the point
$(0,k)=\overline{(a)(b,0)}\cap \overline{\pto\pty}$.
It lies on the line $[a,k]$ by construction.
Furthermore, from the definitinon of the PTR and $\odot$ we see
$k=T(a,b,0)=a\odot b$. Thus $(0,)=(0,a\odot b)$.
\end{itemize}
This action is represented in Figure 5.
\begin{center}
\begin{tikzpicture}[domain=-2:12]
\coordinate 	(O) at (0,0);
\coordinate 	(X) at (10,0);
\coordinate 	(Y) at (0,5);
\coordinate 	(One0) at (2,0);
\coordinate	(Zeroa) at (0,1);
\coordinate	(Zerob) at (0,3.5);
\coordinate	(B0) at (4.2,0);
\coordinate	(Zerok) at (0,2.2);

\draw		(O) circle [radius=3pt] node[anchor=north east]{$\pto$};
\draw		(X) circle [radius=3pt] node[anchor=north]{$\ptx=(0)$};
\draw		(Y) circle [radius=3pt] node[anchor=east]{$\pty=(\infty)$};
\draw		(One0) circle [radius=3pt] node[anchor=north]{$(1,0)$};
\draw		(B0) circle [radius=3pt] node[anchor=north]{$(b,0)$};
\draw		(Zeroa) circle [radius=3pt] node[anchor=east]{$(0,a)$};
\draw		(Zerob) circle [radius=3pt] node[anchor=east]{$(0,b)$};
\draw[fill]	(Zerok) circle [radius=3pt] node[anchor=east]{$(0,a\odot b)$};

\draw[thick]	(-1,0) -- (11,0);
\draw[thick]	(0,-1) -- (0,6);

\draw[thick, name path=XY]	(Y) .. controls (14/3,14/3) and (8,3) .. (X);

\path[name path=Zero1J]		(Zero1) parabola bend (One0) (8.5,2.2);


\draw[name intersections={of=XY and Zero1J, by={J}}]
		(J) circle [radius=3pt] node[anchor=south west]{$\ptj=(1)$};

\path[name path=ZeroaA]		(Zeroa) parabola bend (One0) (7,3);

\draw[name intersections={of=XY and ZeroaA, by={A}}]
		(A) circle [radius=3pt] node[anchor=south west]{$(a)$};

\draw[dotted, thick]	(Zeroa) parabola bend (One0) (A);

\draw[dotted, thick]	(Zerob) parabola bend (B0) (J);

\draw[dashed, very thick]	(Zerok) parabola bend (B0) (A);

\end{tikzpicture}

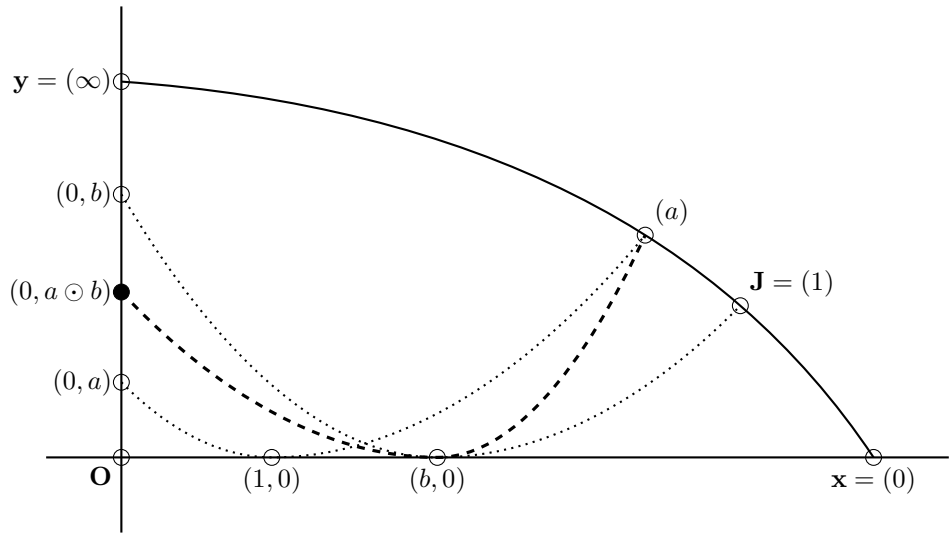
Figure 5: Action of the multiplicative loop on the vertical line.
\end{center}

\subsection{Fano configurations in projective planes}

As an aside before moving to the main motivation for this article, we first
provide a theorem concerning a well known conjecture in projective geometry.
It is possible, perhaps even probable given the simplicity of our argument,
that the main theorem of this section is known, but we have not been able to
locate it.

The Fano configuration must be one of the most oft drawn graphs in all of
mathematics. Here it is (again!):
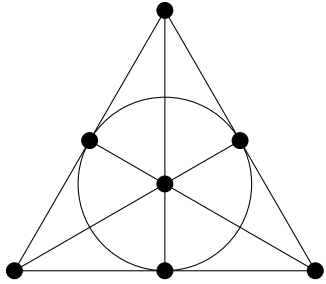
\begin{center}
\begin{tikzpicture}
\filldraw	(0,0) circle [radius=3pt]
		(2,0) circle [radius=3pt]
		(4,0) circle [radius=3pt]
		(2,1.155) circle [radius=3pt]
		(2,3.464) circle [radius=3pt]
		(3,1.732) circle [radius=3pt]
		(1,1.732) circle [radius=3pt];
\draw		(2,1.155) circle [radius=1.154];
\draw		(0,0) -- (3,1.732);
\draw		(4,0) -- (1,1.732);
\draw		(2,0) -- (2,3.464);
\draw		(0,0) -- (2,3.464);
\draw		(4,0) -- (2,3.464);
\draw		(0,0) -- (4,0);
\end{tikzpicture}
\end{center}
Using the action of the additive loop on the vertical line described above,
here we establish a necessary and sufficient condition for any
projective plane to contain a Fano configuration.
\begin{thm} \label{fanoconfig}
Let $\pplane$ be a projective plane of finite order.
Then $\pplane$ contains a Fano configuration if and only if it can be
coordinatised in such a way that the resulting additive loop contains an
involution.
\end{thm}
\begin{proof}
Suppose first that the plane $\pplane$ has been coordinatised in such a way
that the resulting additive loop, $(\R,\oplus)$, contains an involution.
Call it $t$.
Then $t\oplus t=0$, so that $(0,t\oplus t)=(0,0)=\pto$.
In particular, $\pto, (t,t)$ and $\ptj$ are collinear.
If we now return to the diagram describing how addition acts on
$\overline{\pto\pty}$ and redraw, we find we have the following scenario:
\begin{center}
\begin{tikzpicture}[domain=-2:12]
\coordinate 	(O) at (0,0);
\coordinate 	(X) at (10,0);
\coordinate 	(Y) at (0,5);
\coordinate	(t0) at (5,0);
\coordinate	(Zerot) at (0,2.8);

\draw[fill]	(O) circle [radius=3pt] node[anchor=north east]{$\pto$};
\draw		(X) circle [radius=3pt] node[anchor=north]{$\ptx$};
\draw		(Y) circle [radius=3pt] node[anchor=east]{$\pty$};
\draw		(t0) circle [radius=3pt] node[anchor=north]{$(t,0)$};
\draw		(Zerot) circle [radius=3pt] node[anchor=east]{$(0,t)$};

\draw[thick]	(-1,0) -- (11,0);
\draw[thick]	(0,-1) -- (0,6);

\draw[thick, name path=XY]	(Y) .. controls (14/3,14/3) and (8,3) .. (X);

\path[name path=Zero1J]		(Zero1) parabola bend (One0) (9.5,2);

\draw[fill, name intersections={of=XY and Zero1J, by={J}}]
		(J) circle [radius=3pt] node[anchor=south west]{$\ptj=(1)$};

\draw[dotted, thick, name path=ZerotX]
	(Zerot) parabola bend (Zerot) (X);

\draw[dotted, thick, name path=t0Y]
	(Y) .. controls (13/6,23/6) and (23/6,13/6) .. (t0);
		
\draw[fill, name intersections={of=t0Y and ZerotX, by={TT}}]
	(TT) circle [radius=3pt] node[anchor=south west]{$(t,t)$};
		
\draw[dotted, thick]	(Zerot) parabola bend (t0) (J);

\draw[dashed, very thick]	(O) parabola bend (TT) (J);

\end{tikzpicture}

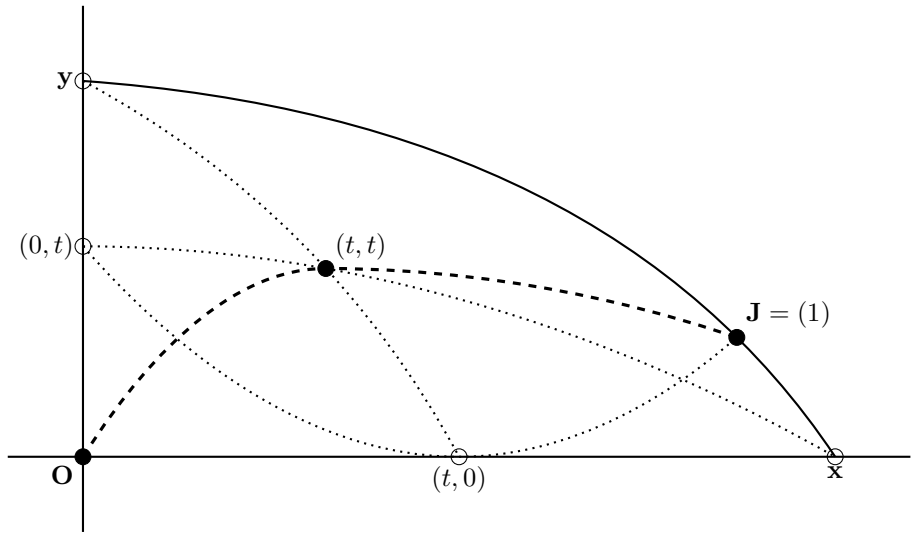
Figure 6: Collinearity of $\pto$, $(t,t)$ and $\ptj$.
\end{center}
This is easily seen to be a Fano configuration.

Conversely, suppose a projective plane $\pplane$ contains a Fano configuration.
Let $\R$ be our coordinatising set and $t\in\Rs$ be fixed.
Choose any triangle $\pto,\ptx,\pty$ of the Fano configuration.
Of the remaining four points in the Fano configuration, three must lie on a
line: label them $(0,t)$, $(t,0)$ and $\ptj=(1)$, with $(0,t)$ on
$\overline{\pto\pty}$, $(t,0)$ on $\overline{\pto\ptx}$ and $\ptj$ on
$\overline{\ptx\pty}$. Finally, label the remaining point $(t,t)$.
If we chose $t=1$, then we have already selected $\pti=(1,1)$ and
coordinatising $\pplane$ using the quadrangle $\pto\ptx\pty\pti$ will result
with $1\oplus 1=0$.
Otherwise, choose an arbitrary point $\pti=(1,1)$ not in the Fano configuration
and set $\overline{\pti\ptj}\cap \overline{\pto\ptx}=(1,0)$ and
$\overline{\pti\ptj}\cap \overline{\pto\pty} = (0,1)$.
Now proceeding to coordinatise $\pplane$ using the quadrangle
$\pto\ptx\pty\pti$, we find $t\oplus t=0$.
In either case we have an involution in $(\R,\oplus)$.
\end{proof}
An immediate corollary of the theorem is the statement concerning
Fano configurations in Desarguesian planes mentioned in the introduction: in a
Desarguesian
plane, any coordinatisation must produce an additive loop that is, in fact,
a group of order equal to the order of the plane. Consequently, the plane must
have even order to allow an involution and no Desarguesian plane of odd order
can contain a Fano configuration.

While the statement gives a clear necessary and sufficient condition, it may
still be viewed as unsatisfying, in that there is no known general criteria
which determine that a loop must contain an involution.

We note that one direction of the above theorem is immediate from the following
general statement.
\begin{thm}
Let $\pplane$ be a projective plane of order $n$ and $\splane$ be a subplane of
$\pplane$ of order $m\le n$.
Let $\R$ be a coordinatising set for $\pplane$ of cardinality $n$.
If the coordinatising quadrangle $\pto\ptx\pty\pti$ is contained in $\splane$,
then there exists a subset $\Sset$ of $\R$ of cardinality $m$ which acts as the
coordinatising set of $\splane$.
Moreover, the PTR $T$ produced by coordinatising $\pplane$ acts as the PTR of
$\splane$ when restricted to $\Sset$.
\end{thm}
The proof follows immediately from the observation that after choosing the
quadrangle $\pto\ptx\pty\pti$ from $\splane$, the sequential way in which
coordinates are assigned guarantees you could simply coordinatise $\splane$
first during the coordinatisation of $\pplane$ (or, indeed, you could just as
easily label the points of $\splane$ last).
Since the coordinatisation of a Desarguesian plane must always produce a 
field under the loop operations arising from the coordinatisation, we get
the following corollary for free.
\begin{cor}
If $\pplane$ is a projective plane of order $n$ containing a Desarguesian
subplane of order $q$, then $\pplane$ can always be coordinatised so that
there is a subset $\Sset$ of the coordinatising set which forms a field of order
$q$ under the operations $\oplus$ and $\odot$ arising from the coordinatisation
of $\pplane$.
\end{cor}

\section{Coordinatising using finite fields} \label{fqcoord}

Throughout the remainder of the paper we fix
$q=p^e$ for some prime $p$ and natural number $e$.
We use $\ff{q}$ to denote the finite field of $q$ elements and $\ffs{q}$ its
non-zero elements.
Every function on $\ff{q}$ can be represented uniquely
by a polynomial in $\ffx{q}$ of
degree less than $q$; this follows at once from Lagrange Interpolation, and
indeed this observation is easily extended to the multivariate case.
Any polynomial whose degree in each variable
is less than $q$ is called {\em reduced}.
A polynomial $f\in\ffxi{q}{0}$ is called a {\em permutation polynomial (PP)}
over $\ff{q}$ if the evaluation map ${\bf x}\mapsto f({\bf x})$ is
equidistributive on $\ff{q}$ -- that is, for each $y\in\ff{q}$, the equation
$f({\bf x})=y$ has $q^{n-1}$ solutions ${\bf x}\in\ff{q}^n$.
(In the case where $n=1$, the evaluation map is a bijection.)
It follows from Hermite's criteria that if a reduced polynomial
$f\in\ffxi{q}{0}$ is a PP over $\ff{q}$, then the degree of $f$ in each $X_i$
is at most $q-2$.
Even a casual perusal of Mathematical Reviews will show PPs have been a
significant research topic in their own right for many years (effectively since
historical times), with a wide array of applications.

A related concept also of interest is that of a $\kappa$-polynomial.
A polynomial $f\in\ffxi{q}{0}$ is a {\em $\kappa$-polynomial over $\ff{q}$} if
\begin{equation*}
k_a = \#\{ \bf x\in\ff{q}^n\,:\, f(\bf x)=a\}
\end{equation*}
is independent of $a$ for $a\in\ffs{q}$.
In direct contrast to the study of permutation polynomials, there are almost no
results in the literature directly discussing $\kappa$-polynomials. This
seems altogether surprising since the specified regularity on preimages of all
non-zero elements of the field suggests such polynomials must almost certainly
appear in many guises.
As an example in how they may arise, recall that a {\em skew Hadamard
difference set (SHDS)} $D\subset \ffs{q}$ is a set of order $(q-1)/2$ where
every element of $\ffs{q}$ can be written as a difference of elements of $D$ in
precisely $(q-3)/4$ ways.
Let $D$ be any SHDS, and define a two-to-one map $\phi:\ffs{q}\rightarrow D$ in
an arbitrary way.
Extending $\phi$ to all of $\ff{q}$ by setting $\phi(0)=0$, we can associate
with $\phi$ a reduced polynomial $f\in\ffx{q}$.
It is straightforward to confirm the polynomial $M(X,Y)=f(X)-f(Y)$ is a
$\kappa$-polynomial over $\ff{q}$ with $k_a=q-1$ for all $a\in\ffs{q}$.
(One could generalise this construction in a suitable way to obtain
$\kappa$-polynomials in more than two variables using difference families.)
The thesis of Matthews, \cite{matthewsphd}, contains some general results on 
$\kappa$-polynomials.
Some of these results are given in the author's Section 9.4 of the Handbook of
Finite Fields \cite{bmullen13}.
Theorem 9.4.8 of \cite{bmullen13}, which is straightforward to prove, shows
how $\kappa$-polynomials play a role in the study of projective planes; the
theorem is extended in Theorem \ref{PACE} below.

One can choose any set $\R$ of cardinality $n$ for the labelling of points in
the coordinatisation process, but since the coordinatisation method will
produce an algebraic structure on the set chosen, there are obviously good and
bad choices.
The resulting function will often exhibit additional algebraic structure,
inherited from the plane, so algebraic sets are obvious candidates.
For example, regardless of the plane, the points $\pto$ and $\pti$ determine
two special elements, zero and one, respectively, of the coordinatisation which
have properties much the same to $0$ and $1$ in any ring with unity.
Since the labelling during the coordinatisation process is arbitrary, by
choosing a ring of order $n$ with unity, we may label the zero and one of the
coordinatisation as the 0 and 1 of the ring.

We now move to make the previous paragraph much more formal in the case where
the plane has prime power order $q$.
Let $\pplane$ be a projective plane of order $q$.
Via coordinatisation, we can obtain a PTR equivalent to the plane $\pplane$.
Since the plane has order $q$, we can view the PTR as some function in three
variables defined over $\ff{q}$, and consequently view the function as a
(reduced) polynomial $T\in\ffxyz{q}$.
Furthermore, since the correspondence of elements in the coordinatisation
and the elements of $\ff{q}$ is arbitrary, we may set the zero and one of the
coordinatisation of $\pplane$ to be the elements $0$ and $1$ of $\ff{q}$.
\begin{defn}
A {\em PTR polynomial $T(X,Y,Z)$ over $\ff{q}$} is any three variabled
polynomial in $\ffxyz{q}$ resulting from the coordinatisation of a plane
$\pplane$ of order $q$ through labelling the points of $\pplane$ using elements
of $\ff{q}$ and where we label $\pto=(0,0)$ and $\pti=(1,1)$.
\end{defn}
Note that for a PTR polynomial, we are guaranteed that the zero and one of the
PTR and the 0 and 1 of $\ff{q}$ coincide.
An equivalent definition is that 
$T\in\ffxyz{q}$ is a PTR polynomial over $\ff{q}$ if it satisfies
Properties (a) through (e) of Lemma \ref{hplem} over $\ff{q}$.

\section{Restrictions on the form of PTR polynomials} \label{ptrpolys}

We now look to exploit the conditions on $T$ described in Lemma \ref{hplem} to
obtain restrictions on the possible forms of $T$.
Throughout we assume $T$ is a reduced polynomial.
\begin{thm} \label{PA}
Suppose $T\in\ffxyz{q}$ satisfies Property (a).
Then
\begin{equation} \label{aneqn}
T(X,Y,Z) = Z + XYZ\, M_1(X,Y,Z) + M_2(X,Y),
\end{equation}
where
\begin{align*}
M_1(X,Y,Z) &= \sum_{i=0}^{q-2} \sum_{j=0}^{q-2} \sum_{k=0}^{q-2}
b_{ijk} X^iY^jZ^k,\\
M_2(X,Y) &= \sum_{i=1}^{q-1} \sum_{j=1}^{q-1} c_{ij} X^i Y^j.
\end{align*}
In particular, 
\begin{equation} \label{odotop}
x\odot y = T(x,y,0) = M_2(x,y)
\end{equation}
for all $x,y\in\ff{q}$.
\end{thm}
\begin{proof}
As a polynomial, we may represent $T$ as
\begin{equation*}
T(X,Y,Z) = \sum_{i,j,k=0}^{q-1} a_{ijk} X^i Y^j Z^k.
\end{equation*}
By Property (a), $T(0,0,z)=z$ for all $z$.
Viewing this as a polynomial identity in $Z$ we immediately find
\begin{equation*}
a_{00k} =
\begin{cases}
1 &\text{ if $k=1$,}\\
0 &\text{ if $k\ne 1$.}
\end{cases}
\end{equation*}
Noting $T(x,0,Z)=Z$ for all $x$, we again view this as a polynomial identity in
$X,Z$, and obtain
\begin{equation*}
Z = T(X,0,Z) = \sum_{i=0}^{q-1} X^i \left(\sum_{k=0}^{q-1} a_{i0k} Z^k\right).
\end{equation*}
For $i\ne 0$, this now forces
\begin{equation*}
\sum_{k=0}^{q-1} a_{i0k} Z^k = 0.
\end{equation*}
As a polynomial identity, we get $a_{i0k}=0$ for all $i\ne 0$.
A similar argument shows $a_{0jk}=0$ for all $j\ne 0$.
Hence
\begin{equation} \label{Teqn}
T(X,Y,Z) = Z + \sum_{i,j=1}^{q-1} \sum_{k=0}^{q-1} a_{ijk} X^iY^jZ^k
= Z + XY \, T_1(X,Y,Z),
\end{equation}
for some reduced $T_1\in\ffxyz{q}$.
It is clear we can now rewrite $T_1$ as claimed in (\ref{aneqn}).
\end{proof}
So we see that Property (a) alone isolates the behaviour of $\odot$, though
of course it does not {\em define} the behaviour of $\odot$.

We now derive a result on PPs; though this could just as easily be established
by considering the plane directly, we choose instead to use as few of the
properties of Lemma \ref{hplem} as is necessary in each case.
\begin{thm} \label{PPthm}
Let $T\in\ffxyz{q}$.
The following statements hold.
\begin{enumerate}
\renewcommand{\labelenumi}{(\roman{enumi})}
\item Suppose $T$ satisfies Properties (a) and (c).
Then $T(X,y,z)$ is a PP in $X$ for every choice of
$(y,z)\in\ffs{q}\times\ff{q}$.

\item Suppose $T$ satisfies Properties (a) and (e).
Then $T(x,Y,z)$ is a PP in $Y$ for every choice of
$(x,z)\in\ffs{q}\times\ff{q}$.

\item Suppose $T$ satisfies Property (d).
Then $T(x,y,Z)$ is a PP in $Z$ for every choice of
$(x,y)\in\ff{q}\times\ff{q}$.
\end{enumerate}
\end{thm}
\begin{proof}
For the 1st claim, an appeal to Property (c) with $c=0\ne a$, $b,d$ arbitrary
shows the equation $T(x,a,b)=T(x,0,d)$ has a unique solution $x$.
By Property (a), $T(x,0,d)=d$, and so $T(x,a,b)=d$ has a unique solution $x$
for each $d\in\ff{q}$.

For (ii), fix $a=0$.
By Property (e), for any $b,c,d$ with $c\ne 0$ there exists a unique $(y,z)$
such that $T(0,y,z)=b$ and $T(c,y,z)=d$.
By Property (a), $T(0,y,z)=z$, and so $z$ is fixed: $z=b$.
Thus, as we range over all $d\in\ff{p}$, we have a unique
preimage $y$, proving the claim.

For (iii), fix $x,y$.
Property (d) tells us that for any $c$, we can always solve uniquely for $z$ in
$T(x,y,z)=c$.
Thus $T(x,y,z_1)=T(x,y,z_2)$ implies $z_1=z_2$, so that $T(x,y,Z)$ is a PP in
$Z$ for every $x,y$.
\end{proof}
\begin{cor} \label{PACDE}
Suppose $T\in\ffxyz{q}$ satisfies Properties (a), (c), (d) and (e).
Then $T$ has degree at most $q-2$ in each of $X$, $Y$, and $Z$.
\end{cor}
\begin{proof}
By assumption, $T$ has the form given in (\ref{aneqn}).
Since $T(x,y,Z)$ is a PP for all $x,y\in\ff{q}$, Hermite's criteria tells us
\begin{equation*}
\sum_{i,j=1}^{q-1} a_{ij\, (q-1)} x^i y^j =0
\end{equation*}
for all $x,y$.
This holds as a polynomial identity in $X,Y$, and so $b_{ij\, (q-1)}=0$
for all $i,j$.
Similar arguments can be obtain the bounds on the degrees of $X$ and $Y$.
\end{proof}
While it may be tempting to surmise from the above that $T(X,Y,z)$ is a PP for
all $z\in\ff{q}$, it is not actually true, as the next result shows.
\begin{thm} \label{PACE}
Suppose $T\in\ffxyz{q}$ satisfies Property (a) and one of Properties (c) or
(e).
Then $T(X,Y,z)-z$ is a $\kappa$-polynomial for any $z\in\ff{q}$.
\end{thm}
\begin{proof}
Fix $z$ and consider the polynomial $f_z\in\ffxy{q}$ given by
$f_z(X,Y)=T(X,Y,z)$.
If $d=z$, then by Property (a), $T(0,y,z)=T(x,0,z)=d$ for all $x,y\in\ff{q}$.
Thus $f_z(x,y)=d$ has (at least) $2q-1$ solutions.
If $d\ne z$, then by Theorem \ref{PPthm} (i) or (ii), there are precisely
$q-1$ solutions $(x,y)\in\ffs{q}\times\ffs{q}$ to the equation $f_z(x,y)=d$.
Since this accounts for all $q^2$ images, we see
\begin{equation*}
f_z(x,y) = d \quad
\begin{cases}
&\text{has $q-1$ solutions when $d\ne z$,}\\
&\text{has $2q-1$ solutions when $d = z$.}
\end{cases}
\end{equation*}
Consequently, the polynomial $f_z(X,Y)-z=T(X,Y,z)-z$ is a $\kappa$-polynomial
over $\ff{q}$.
\end{proof}
\begin{cor} \label{fullPP}
Suppose $T\in\ffxyz{q}$ satisfies either
Property (a) and one of Properties (c) or (e); or
Property (d).
Then $T(X,Y,Z)$ is a PP over $\ff{q}$
\end{cor}
\begin{proof}
Suppose first that $T$ satisfies Property (a) and one of Properties (c) or (e).
Fixing $z,d\in\ff{q}$, we see from the proof of Theorem \ref{PACE} that
\begin{equation*}
T(x,y,z) = d \quad
\begin{cases}
&\text{has $q-1$ solutions when $z\ne d$,}\\
&\text{has $2q-1$ solutions when $z = d$.}
\end{cases}
\end{equation*}
Consequently, as we range over all $z\in\ff{q}$, a given $d$ has
$(2q-1)+(q-1)(q-1)=q^2$ preimages $(x,y,z)\in\ff{q}^3$.

Now suppose Property (d) is satisfied.
Then by Theorem \ref{PPthm} (iii), $T(x,y,Z)$ is a PP for all choices of
$(x,y)\in\ff{q}\times\ff{q}$.
It follows at once that $T(x,y,z)=d$ has precisely $q^2$ solutions
$(x,y,z)$.
\end{proof}
At this point, we have shown that Properties (a), (c), and (d) can lead to
PPs.
Property (e) can also be used to derive a PP result, but not over $\ff{q}$.
Suppose $T\in\ffxyz{q}$.
Let $\{1,\beta\}$ be a basis for $\ff{q^2}$ over $\ff{q}$.
For any $a,b\in\ff{q}$, we define the function
$S_{a,b}:\ff{q^2}\rightarrow\ff{q^2}$ by
\begin{equation*}
S_{a,b}(x) = S_{a,b}(y+\beta z) = T(a,y,z) + \beta T(b,y,z).
\end{equation*}
When we talk of the polynomial $S_{a,b}$ we will mean the polynomial of least
degree in $\ffx{q^2}$ which when induced produces the function just defined.
The following lemma is now immediate.
\begin{lem} \label{PE}
Suppose $T\in\ffxyz{q}$ satisfies Property (e).
Then $S_{a,b}$ is a permutation polynomial over $\ff{q^2}$ whenever $a\ne b$.
\end{lem}
Finally we move to consider how Property (b) impacts the form of the PTR
polynomial.
We have already seen how Property (a) alone isolates the behaviour of $\odot$,
see (\ref{odotop}) above.
One interesting outcome of combining Properties (a) and (b) is that the
behaviour of $\oplus$ is also isolated.
\begin{lem} \label{PAB}
Suppose $T\in\ffxyz{q}$ satisfies Properties (a) and (b).
Then $T$ has the shape (\ref{aneqn}) and
\begin{equation*}
\sum_{i=1}^{q-1} c_{ij} =
\sum_{i=1}^{q-1} c_{ji} =
\begin{cases}
1 &\text{ if $j=1$,}\\
0 &\text{ if $j>1$.}
\end{cases}
\end{equation*}
Moreover,
\begin{equation} \label{oplusop}
y\oplus z = T(1,y,z) = y+z + yz \, M_1(1,y,z)
\end{equation}
for all $y,z\in\ff{q}$.
\end{lem}
\begin{proof}
From Property (b), we know $T(X,1,0)=X$.
Combining this polynomial identity with (\ref{aneqn}) forces the first
set of conditions on the coefficients, while using $T(1,Y,0)=Y$ forces the
second set.
In addition, applying $T(1,y,0)=y$ to (\ref{aneqn}), we also find
$T(1,y,z) = y+z + yz \, M_1(1,y,z)$, as claimed.
\end{proof}
Now, if we combine all of the above, we obtain the following result about
PTR polynomials, the proof of which is immediate from the above statements.
\begin{thm} \label{ptrpolyform}
Suppose $T(X,Y,Z)$ is a PTR polynomial over $\ff{q}$.
Then
\begin{equation} \label{aneqn2}
T(X,Y,Z) = Z + XYZ\, M_1(X,Y,Z) + M_2(X,Y),
\end{equation}
with
\begin{align*}
M_1(X,Y,Z) &= \sum_{i=0}^{q-3} \sum_{j=0}^{q-3} \sum_{k=0}^{q-3}
b_{ijk} X^iY^jZ^k,\\
M_2(X,Y) &= \sum_{i=1}^{q-2} \sum_{j=1}^{q-2} c_{ij} X^i Y^j.
\end{align*}
In addition, $T$ is linear if and only if for all $x,y,z\in\ff{q}$, $z\ne 0$,
we have
\begin{equation} \label{keyidentity}
xy\, M_1(x,y,z) = M_2(x,y) M_1(1,M_2(x,y),z).
\end{equation}
\end{thm}
We get an immediate corollary which extends Lemma \ref{PAB} for linear PTR
polynomials.
\begin{cor}
For a linear PTR polynomial $T\in\ffxyz{q}$ of the form (\ref{aneqn2}),
we have
\begin{equation*}
\sum_{i=0}^{q-3} b_{ijk} = \sum_{i=0}^{q-3} b_{jik}
\end{equation*}
for all $0\le j\le q-3$ and $1\le k\le q-3$.
\end{cor}
The result follows by substituting $y=1$ into (\ref{keyidentity}), whereby one
obtains $xM_1(x,1,z)=xM_1(1,x,z)$ for all $x,z\in\ff{q}$.
This can be viewed as a polynomial equation in $X,Z$ and the statement of the
corollary follows.

\section{On the Lenz-Barlotti classification and coordinatisation}
\label{LBsectionpt1}

Let $\pplane$ be a projective plane and $\Gamma$ denote the full
collineation group of $\pplane$.
If a collineation fixes a line $\lne$ pointwise and a point $\pta$ linewise,
then it is called a {\em central} collineation, and $\lne$ and $\pta$ are
called the {\em axis} and {\em center} of the collineation, respectively.
It is well known that every central collineation in $\Gamma$ has a unique
center $\pta$ and unique axis $\lne$.
Let $\Gamma(\pta,\lne)$ be the subgroup of $\Gamma$ consisting of all
central collineations of $\pplane$ with center $\pta$ and axis $\lne$.
The plane $\pplane$ is said to be {\em $(\pta,\lne)$-transitive} if for every
two distinct points $\ptb,\ptc$ that are (a) collinear with $\pta$ but not
equal to $\pta$, and (b) not on $\lne$, there exists a necessarily unique
collineation $\gamma\in\Gamma(\pta,\lne)$ which maps $\ptb$ to $\ptc$.
Now let $\mne$ be a second line of $\pplane$, not necessarily distinct from
$\lne$.
If $\pplane$ is $(\pta,\lne)$-transitive for all $\pta\in\mne$, then
$\pplane$ is said to be {\em $(\mne,\lne)$-transitive}; the concept of {\em
$(\pta,\ptb)$-transitivity} is defined dually.
If $\pplane$ is $(\lne,\lne)$-transitive, then $\lne$ is called a {\em
translation line} and $\pplane$ is called a {\em translation plane} with
respect to the line $\lne$.
The definitions of {\em translation point} and {\em dual translation plane} are
defined dually also.

The Lenz-Barlotti (LB) classification for projective planes is based on the
possible sets
$$\T=\{ (\pta,\lne) \,:\, \pplane \text{ is $(\pta,\lne)$-transitive} \}$$
of point-line transitivities that the full collineation group of a plane can
exhibit.
Developed by Lenz \cite{len54} and refined by Barlotti \cite{bar57}, the
classification has a heirarchy of types, starting with little to no point-line
transitivities in types I and II, through to type VII.2, which represents the
Desarguesian plane and where $\T$ consists of every possible point-line flag.
There are no type VI planes at all -- the type arises naturally in the study
of potential permutation groups, but no plane can exist of this type.
For any LB type where a finite example is known, one can also find an infinite
example.
The converse is not true; infinite examples of types III.1, III.2 and VII.1 are
known, while it can be shown that finite examples of each of these types are
impossible -- in the case of type VII.1, this is due to the
Artin-Zorn Theorem which states any finite alternative division ring is a
field, see \cite{bhughes73}, Theorem 6.20; type III.1
was ultimately resolved by Hering and Kantor \cite{hering71} and type III.2 was
completed by L{\" u}neberg \cite{lun65} and Yaqub \cite{yaq67}.
It should be noted that several finite cases remain open --
the question of
existence of finite projective planes of LB types I.2, I.3, I.4 and
II.2 remains unresolved.

Our motivation for discussing the Lenz-Barlotti types of projective planes
is made clear when we return to considering the coordinatisation of planes.
In parallel with the Lenz-Barlotti classification, there is a corresponding
structural heirarchy for properties of PTRs as one ascends through the
Lenz-Barlotti types, though one now assumes that the coordinatisation is done
in such a fashion so that the resulting PTR exhibits the most structure.
In LB type I.1, the PTR has no additional structure beyond Lemma \ref{hplem}.
All other planes can be coordinatised to produce a linear PTR.
A LB type II plane can be coordinatised to produce a PTR $T$ which is linear
and where $\oplus$ is associative (so $\oplus$ describes a group operation on
the coordinatising set $\R$).
Any plane which is at least LB type IV is a translation plane.
LB type IV planes can be coordinatised to produce quasifields, LB type V planes
can produce semifields, and the Desarguesian case, of course, can produce a
field.
More specifically, we can say the following.
\begin{lem} \label{lblem}
The following statements hold.
\begin{enumerate}
\renewcommand{\labelenumi}{(\roman{enumi})}

\item
A plane $\pplane$ which is {\em only} $((0),[0])$-transitive is necessarily
LB type I.2.
The plane $\pplane$ is $((0),[0])$-transitive if and only if it can be
coordinatised by a linear PTR with associative multiplication $\odot$.
In such cases, $\Gamma((0),[0])$ is isomorphic to the group described by
$\odot$.
Moreover. during coordinatisation, $\ptx$ is chosen to be the point $(0)$.

\item
A plane $\pplane$ which is {\em only} $((0),[0])$-transitive and
$((\infty),[0,0])$-transitive is necessarily LB type I.3.
The plane $\pplane$ is $((0),[0])$-transitive and $((\infty),[0,0])$-transitive
if and only if it can be coordinatised by a linear PTR with associative
multiplication $\odot$ and displaying a left distributive law.

\item
A plane $\pplane$ which is $((\infty),[\infty])$-transitive is necessarily 
LB type at least II.
The plane $\pplane$ is $((\infty),[\infty])$-transitive if and only if 
it can be coordinatised by a linear PTR with associative addition $\oplus$.
In such cases, $\Gamma((\infty),[\infty])$ is isomorphic to the group described
by $\oplus$.
Moreover, during coordinatisation, $\pty$ is chosen to be the point $(\infty)$.

\item
A plane $\pplane$ which is a translation plane or dual translation plane is
necessarily Lenz-Barlotti type at least IV.
The plane $\pplane$ is a translation plane (resp. dual translation plane) if
and only if it can be coordinatised by a linear PTR with associative addition
$\oplus$ and a right distributive law $(x\oplus y)\odot z=x\odot z + y\odot z$
(resp. a left distributive law $x\odot (y\oplus z)=x\odot y + x\odot z$).
In such cases, the order of $\pplane$ must be a prime power $q$ and the group
described by $\oplus$ is elementary abelian.
Moreover, during coordinatisation, $\overline{\ptx\pty}$ is the translation
line (resp. $\pty$ is the translation point).

\item
A plane $\pplane$ which is both a translation plane and a dual translation
plane (so $[\infty]$ is a translation line and $(\infty)$ is a translation
point) is necessarily Lenz-Barlotti type at least V.
The plane $\pplane$ is LB type at least V if and only if it can be
coordinatised by a linear PTR with associative addition $\oplus$ and both a
left and right distributive law.
In such cases, the order of $\pplane$ must be a prime power $q$ and the group
described by $\oplus$ is elementary abelian.
Moreover, during coordinatisation, the $\overline{\ptx\pty}$ is the translation
line and $\pty$ is the translation point.
\end{enumerate}
\end{lem}
These results come from \cite{bdembowski68}, Chapter 3, and \cite{bhughes73},
Chapters 5 and 6, and we refer the reader to these references for further
information on the Lenz-Barlotti classification and the corresponding
properties of PTRs.

This leaves open one obvious question, that of how to coordinatise a plane
optimally.
Lemma \ref{lblem} makes clear the following strategy to be used during the
coordinatisation process:
\begin{itemize}
\item If $\T$ contains an incident point-line flag, one such flag
must be $((\infty),[\infty])$.
\item If $\T$ contains a non-incident point-line flag, one such flag
must be $((0),[0])$.
\end{itemize}
Unless the plane is LB type I.1, at least one, and possibly both, of these
strategems can be met during the initiation phase of the
coordinatising process, when one chooses the triangle $\pto\ptx\pty$.
{\em In the following, we assume that the planes have been coordinatised
optimally with respect to the properties exhibited by the PTR, and in
accordance with the above strategy.}
As part of such an ``optimising" strategy, we prioritise associativity of the
operations $\oplus$ and $\odot$ of the PTR over distributivity whenever there
is such a choice available.

This optimal coordinatisation can be exploited even further through the use of
Figures 4 or 5.
For example, if $\pplane$ is $((\infty),[\infty])$-transitive and the group
$\Gamma((\infty),[\infty])$ is known, one can use that group as the labelling
set and use Figure 4 to ensure that $\oplus$ is actually the operation of the
group.
Likewise, if the plane is $((0),[0])$-transitive and the group
$\Gamma((0),[0])$ is known, one can use that group, along with an additional
element 0, as the labelling set and use Figure 5 to ensure that $\odot$ is
actually the operation of the group.
It is for this specific reason that we have taken such care in describing the
coordinatisation method and the actions of the two loops on the vertical line
in Section 2 -- if these actions were not able to be described explicitly, then
one could not pursue the optimal coordinatising strategy we've outlined. 

Linking these optimising strategies to PTR polynomials, the most obvious cases
we might be interested in is when either $\Gamma((\infty),[\infty])$ is
elementary abelian, or when $\Gamma((0),[0])$ is cyclic.
In the former case, through optimal coordinatisation, we can assume $\oplus$ 
is field addition, while in the latter case, we can force $\odot$ to be field
multiplication through coordinatising optimally.
(It should be noted that one cannot simultaneously assume optimal
coordinatisation for both $\oplus$ and $\odot$ as the labelling of the line
$\overline{\pto\pty}$ is determined by exactly one of Figures 4 or 5 in these
optimising strategies.)
In cases where neither of these conditions arise, a representation theory for
representing groups by polynomials is needed; such a theory was recently
developed by Castillo and the author, see \cite{coulter15}.

For the remainder of this article, we consider how knowing either $\oplus$ or
$\odot$ is a field operation affects the PTR polynomial.
We begin first with the case where $\oplus$ is assumed to be field addition --
this situation is actually quite common, especially in the study of semifields,
dating back to the first proper examples given by Dickson in \cite{dic06}.
In fact, if the plane is Lenz-Barlotti IV or higher, then you are guaranteed
that any optimal coordinatisation will force $\oplus$ to be field addition.
\begin{thm} \label{realaddition}
Let $\pplane$ be a projective plane of order $q=p^e$ for some prime $p$ which
is $((\infty),[\infty])$-transitive and where $\Gamma((\infty),[\infty])$ is
elementary abelian.
Suppose $T\in\ffxyz{q}$ is a PTR polynomial obtained from coordinatising
$\pplane$ optimally, so that the resulting additive loop is field addition.
\begin{enumerate}
\renewcommand{\labelenumi}{(\roman{enumi})}
\item If $\pplane$ is strictly LB type II.1, then 
\begin{equation} \label{basicform}
T(X,Y,Z)=M_2(X,Y)+Z,
\end{equation}
where $M_2(X,Y)$ is as in (\ref{aneqn2}).

\item If $\pplane$ is strictly LB type II.2, then 
$T\in\ffxyz{q}$ is of the shape (\ref{basicform}) and where
\begin{equation*}
M_2(x,M_2(y,z))=M_2(M_2(x,y),z)
\end{equation*}
for all $x,y,z\in\ff{q}$.

\item If $\pplane$ is a translation plane of LB type at least IV, then
$T\in\ffxyz{q}$ is of the shape (\ref{basicform}) and where
\begin{equation} \label{lbiv}
M_2(X,Y)=\sum_{i=0}^{e-1}\sum_{j=1}^{q-1} c_{ij} X^{p^i} Y^j.
\end{equation}

\item If $\pplane$ is a dual translation plane of LB type at least IV, then
$T\in\ffxyz{q}$ is of the shape (\ref{basicform}) and where
\begin{equation} \label{lbivd}
M_2(X,Y)=\sum_{i=1}^{q-1}\sum_{e=0}^{e-1} c_{ij} X^i Y^{p^j}.
\end{equation}

\item If $\pplane$ is LB type at least V, then $T\in\ffxyz{q}$ is of the shape
(\ref{basicform}) and where
\begin{equation} \label{lbv}
M_2(X,Y)=\sum_{i=0}^{e-1}\sum_{j=0}^{e-1} c_{ij} X^{p^i} Y^{p^j}.
\end{equation}

\end{enumerate}
\end{thm}
\begin{proof}
By our hypotheses, the plane $\pplane$ is necessarily LB type at least II.1,
and $y\oplus z=y+z$, so that in (\ref{oplusop}) we see $M_1=0$.
The claim of (i) now follows at once from Theorem \ref{ptrpolyform}.
Extending to LB type II.2 is immediate from the fact that, in an optimal
coordinatisation, the plane will be both $((\infty),[\infty])$-transitive and
$((0),[0])$-transitive, and $x\odot y = M_2(x,y)$ will act isomorphically to
$\Gamma((0),[0])$. Thus the condition given on $M_2$ is nothing more than the
associative property of the operation $\odot$.

For (iii), Lemma \ref{lblem} tells us we must have Equation \ref{basicform}, as
well as a right distributive law.
Thus $M_2(X,Y)$ must satisfy $M_2(a+b,y)=M_2(a,y)+M_2(b,y)$ for all
$a,b,y\in\ff{q}$.
It follows at once that $M_2(X,Y)$ is a linearised polynomial in $X$.
Thus $M_2$ has the form claimed.
A similar argument deals with the case (iv).
The claims of (v) now follow at once as a LB type V plane is both
a translation plane and a dual translation plane.
\end{proof}
It is worth noting that whenever we consider a projective plane of LB type at
least IV, we are guaranteed that we can obtain a PTR polynomial of one of the
shapes (\ref{lbiv}), (\ref{lbivd}), or (\ref{lbv}), via Lemma \ref{lblem}.

A polynomial $f\in\ffx{q}$ is called a {\em complete mapping on $\ff{q}$} if
both $f(X)$ and $f(X)+X$ are PPs over $\ff{q}$.
Complete mappings and their extensions have been studied in several situations.
For example, they are connected to the construction of latin squares.
Our next result shows how complete mappings arise
completely naturally and in numbers when we look at PTR polynomials.
\begin{lem}
Let $\pplane$ be a projective plane of order $q=p^e$ for some prime $p$ which
is $((\infty),[\infty])$-transitive and where $\Gamma((\infty),[\infty])$ is
elementary abelian.
Suppose $T\in\ffxyz{q}$ is a PTR polynomial obtained from coordinatising
$\pplane$ optimally, so that the resulting additive loop is field addition.
Then, for any $a\in\ff{q}\setminus\{0,1\}$, the polynomial $f_a(X)=M_2(X,a)-X$,
is a complete mapping on $\ff{q}$.
\end{lem}
\begin{proof}
By Theorem \ref{realaddition}, we know $T(X,Y,Z)=M_2(X,Y)+Z$.
We now appeal to Properties (b) and (c).
By Property (c), for $a,b,c,d\in\ff{q}$ with $a\ne c$, there exists a unique
$x$ satisfying $M(x,a)+b=M(x,c)+d$.
Setting $b=0$, $c=1$ and appealing to Property (b), we find for all $a\ne 1$,
$M(x,a)-M(x,1)=M(x,a)-x=d$ has a unique solution in $x$ for any $d$.
Thus $f_a(X)=M(X,a)-X$ is a permutation polynomial over $\ff{q}$ for all
$a\ne 1$.
Additionally, $f_a(X)+X=M(X,a)=T(X,a,0)$ is a permutation polynomial for all
$a\ne 0$ by Theorem \ref{PPthm} (i).
\end{proof}

It remains to consider what can be said about PTR polynomials when we know
$\odot$ coincides with field multiplication.
Our initial assumption, then, must be that the plane is at least
$((0),[0])$-transitive.
We note that in this case, by starting with a finite projective plane with a
non-incident flag transitivity, the only LB types possible are I.2, I.3, I.4,
II.2, the planar nearfields of type IV, or VII.2
We may ignore the planar nearfields case, as the multiplicative groups involved
in that case are necessarily non-abelian, so can never be cyclic.
Since II.2 strictly contains only I.2, in the heirarchy of LB types under
consideration, we have two distinct strings:
\begin{itemize}
\item I.2 $\subseteq$ I.3 $\subseteq$ I.4 $\subseteq$ VII.2, and
\item I.2 $\subseteq$ II.2 $\subseteq$ VII.2.
\end{itemize}
Furthermore, it was shown by Ghinelli and Jungnickel \cite{ghinelli03} that I.3
and I.4 planes correspond to the non-abelian and abelian case, respectively, of
the same existence problem for neo-difference sets.
In particular, by assuming $x\odot y=xy$, when we come to consider classes I.3
and I.4, we are enforcing the abelian case; this is why LB type I.3 does not
occur in the following statement.
\begin{thm} \label{realmultiplication}
Let $\pplane$ be a projective plane of order $q=p^e$ for some prime $p$ which
is $((0),[0])$-transitive and where $\Gamma((0),[0])$ is cyclic.
Suppose $T\in\ffxyz{q}$ is a PTR polynomial obtained from coordinatising
$\pplane$ optimally, so that the resulting multiplicative loop is field
multiplication.
\begin{enumerate}
\renewcommand{\labelenumi}{(\roman{enumi})}
\item If $\pplane$ is strictly LB I.2, then 
\begin{equation} \label{lbi2eq}
T(X,Y,Z)= Z + XY + XYZ\, M_1(X,Y,Z),
\end{equation}
where
\begin{equation*}
M_1(X,Y,Z)=\sum_{i,j=0}^{q-3} b_{ij} (XY)^i Z^j.
\end{equation*}

\item If $\pplane$ is strictly LB I.4, then $T$ is of the shape (\ref{lbi2eq})
and where
\begin{equation*}
M_1(X,Y,Z)=\sum_{i=0}^{q-3} b_i (XY)^i Z^{q-2-i}.
\end{equation*}

\item If $\pplane$ is strictly LB II.2, then $T$ is of the shape (\ref{lbi2eq})
and where
\begin{align*}
yz+xy\, M_1(1,x,y)&\left(1+z\, M_1(1,x+y+xy\, M_1(1,x,y),z)\right)\\
&= xy+yz\, M_1(1,y,z)\left(1+x\, M_1(1,x,y+z+yz\, M_1(1,y,z))\right)
\end{align*}
for all $x,y,z\in\ff{q}$.
\end{enumerate}
\end{thm}
\begin{proof}
By hypothesis, $x\odot y=xy$, and Lemma \ref{lblem} tells us the PTR is linear.
Thus $T(x,y,z)=(xy)\oplus z$, and now an appeal to Theorem \ref{ptrpolyform}
produces the first claim, where we define $b_{ij}$ by $b_{ij}=b_{iij}$.

For the second, we use the fact the PTR polynomial $T$ obtained from optimal
coordinatisation must have a left distributive law.
Since $x(y\oplus z)=xy\oplus xz$ for all $x,y,z\in\ff{q}$, we have the
identity
\begin{equation*}
xyz\, M_1(1,y,z) = x^2 yz\, M_1(1,xy,xz)
\end{equation*}
for all $x,y,z$.
Now this equation has no higher powers of $y$ or $z$ beyond the $(q-2)$nd, and
so we can view this as a polynomial identity in $Y,Z$.
Equating coefficients, we find for all $x\in\ff{q}$ and all $0\le i,j\le q-3$,
\begin{equation*}
b_{ij} x = b_{ij} x^{2+i+j}.
\end{equation*}
Thus $b_{ij}=0$ unless $2+i+j=q$, which proves we may index the (potentially) 
non-zero coefficients by a single counter, and this yields the 2nd claim.

For (iii), the proof is essentially the same as for LB type II.2 in Theorem
\ref{realaddition}, in that we know $\oplus$ will be associative in an optimal
coordinatisation of the plane $\pplane$ and the condition on $M_1$ given above
is equivalent.
\end{proof}


\providecommand{\bysame}{\leavevmode\hbox to3em{\hrulefill}\thinspace}
\providecommand{\MR}{\relax\ifhmode\unskip\space\fi MR }
\providecommand{\MRhref}[2]{%
  \href{http://www.ams.org/mathscinet-getitem?mr=#1}{#2}
}
\providecommand{\href}[2]{#2}

\end{document}